\documentclass{amsart}
\usepackage[dvipdfm]{graphicx} 
\usepackage[initials]{amsrefs}

\usepackage{amsthm}
\usepackage{amsmath}
\usepackage{amssymb}
\usepackage{graphicx}

\usepackage{color}
\usepackage{ulem}

\numberwithin{equation}{section}

\theoremstyle{plain}	     
\newtheorem{thm}{Theorem}[section] 
\newtheorem{cor}[thm]{Corollary}
\newtheorem{lem}[thm]{Lemma}

\theoremstyle{definition} 

\newtheorem{exmp}[thm]{Example}

\theoremstyle{remark} 
\newtheorem{rem}[thm]{Remark}
	
\newcommand{\sn}{\operatorname{sn}}
\newcommand{\cn}{\operatorname{cn}}
\newcommand{\dn}{\operatorname{dn}}
\newcommand{\am}{\operatorname{am}}

\newcommand{\R}{\mathbb{R}}

\begin{document}

\title[generalized Jacobi elliptic functions]{Properties of generalized Jacobi elliptic functions with three parameters}

\author{Hajime Sato}
\address[Hajime Sato]{Department of Mathematical Sciences, Shibaura Institute of Technology, 307 Fukasaku, Minuma-ku, Saitama-shi, Saitama 337-8570, Japan}

\author{Nagi Suzuki}
\address[Nagi Suzuki]{Department of Mathematical Sciences, Shibaura Institute of Technology, 307 Fukasaku, Minuma-ku, Saitama-shi, Saitama 337-8570, Japan}

\author{Shingo Takeuchi}
\address[Shingo Takeuchi]{Department of Mathematical Sciences, Shibaura Institute of Technology, 307 Fukasaku, Minuma-ku, Saitama-shi, Saitama 337-8570, Japan}
\email[Corresponding author]{shingo@shibaura-it.ac.jp}
\thanks{The work of S. Takeuchi was supported by JSPS KAKENHI Grant Number 22K03392.}

\dedicatory{Dedicated to Professor Yoshitsugu Kabeya 
on the occasion of his 60th birthday}

\begin{abstract}
Jacobi elliptic functions and complete elliptic integrals are generalized using three parameters. These generalized functions and integrals are closely related to ordinary differential equations involving $p$-Laplacian. In this paper, Wallis-type integral formulae are constructed for the generalized Jacobi elliptic functions. Moreover, for the generalized complete elliptic integrals, a Legendre-type relation is derived, which is equivalent to Elliott's identity for Gaussian hypergeometric series, along with its implications. In addition, nontrivial inequalities on binomial expansions of generalized Jacobi elliptic functions are given.
\end{abstract}


\maketitle


\section{Introduction}

Jacobi elliptic functions and complete elliptic integrals are very useful tools in expressing exact solutions of differential equations. For example, the Jacobi elliptic function $y=\sn{(x,k)}$ with modulus $k \in [0,1)$ and 
the complete elliptic integral of the first kind $K(k)$ satisfy the differential equation and the boundary conditions:
\begin{equation}
\label{eq:ACE}
\begin{cases}
y''+y(1+k^2-2k^2y^2)=0, \\
y(0)=y(2K(k))=0.
\end{cases}
\end{equation}
Nonlinear differential equations of this kind frequently appear in models describing bistable phenomena, and their solutions are expected to be expressible in terms of $\sn{(x,k)}$ and $K(k)$.

In this paper, we consider generalizations of Jacobi elliptic functions 
and of complete elliptic integrals by adding some parameters. The generalized Jacobi elliptic functions (GJEFs) and the generalized complete elliptic integrals (GCEIs) satisfy nonlinear ordinary differential equations involving $p$-Laplacian. For instance, $y=\sn_{p,q,p^*}(x,k)$ and $K_{p,q,p^*}(k)$,
which are defined in Section \ref{sec:preparation}, satisfy
\begin{equation*}
\begin{cases}
(|y'|^{p-2}y')'+\dfrac{(p-1)q}{p}|y|^{q-2}y(1+k^q-2k^q|y|^q)=0,\\
y(0)=y(2K_{p,q,p^*}(k))=0,
\end{cases}
\end{equation*}
where $p, q$ are parameters greater than $1$.
When $p=q=2$, it coincides with \eqref{eq:ACE}.
In fact, the GJEFs we are dealing with involve three parameters and satisfy more general equations, for example, \eqref{eq:snode}--\eqref{eq:dnallencahn} below.

The GJEFs include 
the generalized trigonometric functions (GTFs), which were proposed by Lindqvist \cite{lindqvist1995} and Dr\'{a}bek and Man\'{a}sevich \cite{DM1999}
as the case $k=0$.
In the last decades, 
GTFs have been the subject of active research,
not only in terms of its mathematical properties,
but also for their physical applications.
For instance, GTFs have been used to study nonlinear
spring-mass systems in Lagrangian Mechanics \cite{WLE2012};
to describe human tooth vibration properties \cite{VC2022,CVC2020};
to describe nonlinear oscillations \cite{CVC2020b,GB2019};
to provide the generalized stiffness and mass coefficients for 
power-law Euler-Bernoulli beams \cite{SNW2020};
in quantum gravity \cite{SPC2016};
in quantum graphs \cite{AE2024};
to express some exact solutions of a nonlinear
Schr\"{o}dinger equation \cite{GPPT2025};
and very recently in information theory \cite{PZpreprint}.
GJEFs and GCEIs have also been studied mathematically since they were proposed by the third author \cite{Takeuchi2011}. See, for example, the survey \cite{YHWL2019} by Yin et al. It is expected that GJEFs and GCEIs will be applied to various sciences in the future, just as GTFs have been.

In the following while, two classical topics on Jacobi elliptic functions and complete elliptic integrals will be presented. Specifically, we will discuss the integral value of powers for Jacobi elliptic functions and the Legendre relation for complete elliptic integrals.

First, we discuss the integration of powers for Jacobi elliptic functions.
When calculating the value of norm for a Jacobi elliptic function, it is necessary to integrate the powers of that function. In particular, when the modulus is zero, the Jacobi elliptic function becomes a trigonometric
one, and the integral of powers for the function can be computed by using the famous Wallis integral formula:
\begin{align}
\label{eq:WIF}
\begin{split}
& \int_0^{\pi/2}\sin^{2n}x\,dx
=\int_0^{\pi/2}\cos^{2n}x\,dx
=\frac{(2n-1)!!}{(2n)!!}\frac{\pi}{2},\\
& \int_0^{\pi/2}\sin^{2n+1}x\,dx
=\int_0^{\pi/2}\cos^{2n+1}x\,dx
=\frac{(2n)!!}{(2n+1)!!}.
\end{split}
\end{align}
On the other hand, when the modulus is not zero, it can be computed by the recurrence formula:
$I_n=I_n(k):=\int_0^{K(k)}\sn^n{(x,k)}\,dx$ satisfies
\begin{equation}
\label{eq:Izenka}
(n+3)k^2I_{n+4}-(n+2)(1+k^2)I_{n+2}+(n+1)I_n=0
\quad (n=0,1,2,\ldots).
\end{equation}
This formula with the initial values of $I_n\ (n=0,1,2,3)$ are given 
in \cite[310.00--310.06]{Byrd-Friedman1971}.
However, 
no formula making the integral of powers for Jacobi elliptic functions explicit seems to be known.

Next, we discuss the Legendre relation.
It is well known that the complete elliptic integrals of the first kind $K(k)$ and of the second kind $E(k)$ satisfy the Legendre relation:
\begin{equation}
\label{eq:CLR}
E(k)K'(k)+K(k)E'(k)-K(k)K'(k)=\frac{\pi}{2} \quad
\mbox{for all $k \in (0,1)$},
\end{equation}
where $k':=\sqrt{1-k^2},\ K'(k):=K(k')$ and $E'(k):=E(k')$.
This nontrivial relation appears in all books on elliptic integrals and is also important in applications. For example, in conjunction with the Gauss identity for arithmetic-geometric mean, it is used to derive the Gauss-Legendre algorithm, a fast formula for computing approximations
of $\pi$. 
Legendre is said to have found his identity \eqref{eq:CLR} by chance, but it is now known that \eqref{eq:CLR} implies that the Wronskian of $K$ and $K'$ can be
represented as
\[W(K,K';k)=-\frac{\pi}{2k(1-k^2)}.\]
For this fact see, for example, the book \cite[Lemma 3.2.6, but the coefficient $x(1-x)$ of $y''$ is missing]{AAR1999} by Andrews et al. In fact, the identity corresponding to (1.4) has been obtained for GCEIs in 
\cite{Takeuchi2016b}
by the third author,
but the meaning of the identity remained unclear in that paper.

The purpose of this paper is to generalize Wallis' formula \eqref{eq:WIF} for integrals of Jacobi elliptic functions and Legendre's relation \eqref{eq:CLR}
satisfied by complete elliptic integrals to GJEFs and GCEIs. 
In addition, nontrivial inequalities on binomial expansions 
satisfied by GJEFs are given.
Although there have been numerous studies of GJEFs and GCEIs, 
most of them have two or fewer free parameters (see \cite{YHWL2019})
. There are few studies of those with as many as three free parameters, as far as we know. 
We only know that the third author \cite{Takeuchi2016b} generalizes the Legendre relation, and Yin, Lin and Qi \cite{YLQ} investigate monotonicity, convexity, and inequalities.

This paper is organized as follows. In Section \ref{sec:preparation}, we introduce some known special functions and
generalizations of trigonometric functions, Jacobi elliptic functions, and complete elliptic integrals, which are the subject of our discussion in the present paper. In Section \ref{sec:differential}, we show that each of these generalized functions satisfies certain nonlinear ordinary differential equations. In Section \ref{sec:integral}, we construct Wallis-type integral formulae for powers of GJEFs. As examples, we give some integrals of powers of classical Jacobi elliptic functions. In Section \ref{sec:legendre}, we generalize the Legendre relation satisfied by complete elliptic integrals to GCEIs by giving its implications. In Section \ref{sec:edmunds}, we give nontrivial inequalities on binomial expansions satisfied by GJEFs.

\section{Preliminaries}
\label{sec:preparation}

In this section, we define generalizations of trigonometric functions, Jacobi elliptic functions, and complete elliptic integrals. We also recall Gauss's hypergeometric series and Appell's hypergeometric series.

Let $p, q>1$ and $s^*:=s/(s-1)$ for $s>1$ so that $1/s+1/s^*=1$. We define 
\begin{equation*}
F_{p,q}(x):=\int_0^x \frac{dt}{(1-t^q)^{1/p}} \quad (x \in [0,1]),
\end{equation*}
and
\begin{equation*}
\pi_{p,q}:=2F_{p,q}(1)=
2\int_0^1 \frac{dt}{(1-t^q)^{1/p}}
=\frac{2}{q}B\left(\frac{1}{p^*},\frac{1}{q}\right),
\end{equation*}
where $B$ is the beta function
\[B(x,y):=\int_0^1 t^{x-1}(1-t)^{y-1}\,dt \quad (x,y>0).\]
Since $F_{p,q}$ is an increasing function 
from $[0,1]$ to $[0,\pi_{p,q}/2]$, 
the inverse function of $F_{p,q}$ from $[0,\pi_{p,q}/2]$ to $[0,1]$ 
is denoted by
\begin{equation*}
\sin_{p,q}{x}:=F_{p,q}^{-1}(x).
\end{equation*}
Moreover, we define 
\begin{equation*}
\cos_{p,q}{x}:=(1-\sin_{p,q}^qx)^{1/p}.
\end{equation*}
It is easy to check the properties: for $x \in (0,\pi_{p,q}/2)$,
\begin{align}
& \notag \cos_{p,q}^p{x}+\sin_{p,q}^q{x}=1,\\
&  \notag (\sin_{p,q}{x})'=\cos_{p,q}{x},\\
& \notag (\cos_{p,q}{x})'=-\frac{q}{p}\sin_{p,q}^{q-1}{x}\cos_{p,q}^{2-p}{x},\\
& \label{eq:ev} (\cos_{p,q}^{p-1}{x})'=-\frac{q}{p^*}\sin_{p,q}^{q-1}{x}.
\end{align}

For simplicity of notation, define 
$\pi_p:=\pi_{p,p}$, 
$\sin_p{x}:=\sin_{p,p}{x}$ and $\cos_p{x}:=\cos_{p,p}{x}$.
Readers can find more information about generalized trigonometric functions in the references \cite{DM1999,Kobayashi-Takeuchi,Lang2011,lindqvist1995,YHWL2019}.

Similarly, for $p,q,r>1$ and $k \in [0,1)$, we define
\begin{align*}
H_{p,q,r}(x,k):&=
\int_0^x \frac{dt}{(1-t^q)^{1/p}(1-k^qt^q)^{1-1/r}}\\
&=\int_0^{F_{p,q}(x)}\frac{d\theta}{(1-k^q\sin_{p,q}^q{\theta})^{1-1/r}} \quad (x \in [0,1]),
\end{align*}
and the complete $(p,q,r)$-elliptic integral of the first kind
\begin{align}
\label{eq:kpqr}
K_{p,q,r}(k):&=H_{p,q,r}(1,k)
=\int_0^1 \frac{dt}{(1-t^q)^{1/p}(1-k^qt^q)^{1-1/r}}\\
&=\int_0^{\pi_{p,q}/2}\frac{d\theta}{(1-k^q\sin_{p,q}^q{\theta})^{1-1/r}}. \notag
\end{align}
The generalization of the complete elliptic integral $K(k)$ to three parameters, $K_{p,q,r}(k)$, was first studied in the third author's paper \cite{Takeuchi2016b}.
Since $H_{p,q,r}(\cdot,k)$ is an increasing function 
from $[0,1]$ to $[0,K_{p,q,r}(k)]$, 
the inverse function of $H_{p,q,r}(\cdot,k)$ from $[0,K_{p,q,r}(k)]$ to $[0,1]$ 
is denoted by
\begin{equation*}
\sn_{p,q,r}x=\sn_{p,q,r}(x,k):=H_{p,q,r}^{-1}(x,k).
\end{equation*}
Moreover, we define
\begin{align*}
& \cn_{p,q,r}x=\cn_{p,q,r}(x,k):=(1-\sn_{p,q,r}^q(x,k))^{1/p},\\
& \dn_{p,q,r}x=\dn_{p,q,r}(x,k):=(1-k^q\sn_{p,q,r}^q(x,k))^{1-1/r}.
\end{align*}
The generalizations of Jacobi elliptic integrals to three parameters were first defined in the third author's paper \cite{Takeuchi2013}. 
We call $(p,q,r)$-\textit{elliptic sine}, \textit{-elliptic cosine}, and \textit{-elliptic delta} for 
$\sn_{p,q,r}, \cn_{p,q,r}$, and $\dn_{p,q,r}$, respectively.
It is easy to check the properties: for $x \in (0,K_{p,q,r}(k))$,
\begin{align}
\notag
& \sn_{p,q,r}(x,0)=\sin_{p,q}{x},\ \cn_{p,q,r}{(x,0)}=\cos_{p,q}x,\ \dn_{p,q,r}(x,0)=1,\\
\label{eq:cs}
& \cn_{p,q,r}^p{x}+\sn_{p,q,r}^q{x}=1,\\
\label{eq:ds}
& \dn_{p,q,r}^{r^*}{x}+k^q\sn_{p,q,r}^q{x}=1,\\
\label{eq:sn'}
& (\sn_{p,q,r}{x})'=\cn_{p,q,r}{x}\dn_{p,q,r}{x},\\
\notag
& (\cn_{p,q,r}{x})'=-\frac{q}{p}\sn_{p,q,r}^{q-1}{x}\cn_{p,q,r}^{2-p}{x}\dn_{p,q,r}{x},\\
\notag
& (\dn_{p,q,r}{x})'=-\frac{q}{r^*}k^q\sn_{p,q,r}^{q-1}{x}\cn_{p,q,r}{x}\dn_{p,q,r}^{2-r^*}{x},\\
\label{eq:cnp'}
&
 (\cn_{p,q,r}^{p-1}{x})'=-\frac{q}{p^*}\sn_{p,q,r}^{q-1}{x}\dn_{p,q,r}{x},\\
\label{eq:dnp'}
& (\dn_{p,q,r}^{r^*-1}{x})'=-\frac{q}{r}k^q\sn_{p,q,r}^{q-1}{x}\cn_{p,q,r}{x}.
\end{align}
For example, the formula \eqref{eq:sn'} is obtained by the differential formula for the inverse function as 
$(\sn_{p,q,r}x)'=1/H'_{p,q,r}(\sn_{p,q,r}x,k)=\cn_{p,q,r}x \dn_{p,q,r}x$.

In addition,  
the complete $(p,q,r)$-elliptic integral of the second kind is defined by
\begin{equation}
\label{eq:epqr}
E_{p,q,r}(k):
=\int_0^1 \frac{(1-k^qt^q)^{1/r}}{(1-t^q)^{1/p}}\,dt
=\int_0^{\pi_{p,q}/2}(1-k^q\sin_{p,q}^q{\theta})^{1/r}\, d\theta.
\end{equation}

For simplicity of notation, define 
$K_{p,q}(k):=K_{p,q,q}(k)$, $E_{p,q}(k):=E_{p,q,q}(k)$,
$K_{p}(k):=K_{p,p,p}(k)$, and $E_{p}(k):=E_{p,p,p}(k)$. 

\begin{rem}
We define the $(p,q,r)$-amplitude function $\am_{p,q,r}(\cdot,k):[0,K_{p,q,r}(k)]\to [0,\pi_{p,q}/2]$ by
\begin{gather*}
x=
\int_0^{\am_{p,q,r}{(x,k)}}
\frac{d\theta}{(1-k^q\sin_{p,q}^q{\theta})^{1-1/r}};
\end{gather*}
then
\begin{align*}
&\sn_{p,q,r}(x,k)=\sin_{p,q}(\am_{p,q,r}(x,k)).\\
&\cn_{p,q,r}(x,k)=\cos_{p,q}(\am_{p,q,r}(x,k)),\\
&\dn_{p,q,r}(x,k)=(1-k^q\sin_{p,q}^q(\am_{p,q,r}(x,k)))^{1/r^*}.
\end{align*}
\end{rem}

Finally, let us review some special functions.
The gamma function defined by
\[\Gamma(s):=\int_0^\infty x^{s-1}e^{-x}\,dx\]
is useful in calculating the value of the beta function:
\[B(x,y)=\frac{\Gamma(x)\Gamma(y)}{\Gamma(x+y)}.\]
Gauss's hypergeometric series is defined by
\[F(\alpha,\beta;\gamma;x):=\sum_{n=0}^\infty \frac{(\alpha)_n(\beta)_n}{(\gamma)_nn!}x^n
\quad (|x|<1)\]
and Appell's hypergeometric series (of the first kind) by
\[F_1(\alpha;\beta,\beta';\gamma;x,y)
:=\sum_{m,n=0}^\infty \frac{(\alpha)_{m+n}(\beta)_m(\beta')_n}{(\gamma)_{m+n}m!n!}x^my^n \quad (|x|,|y|<1).
\]
These are undefined if $\gamma$ equals a non-positive integer.
Here, $(\zeta)_n$ is the Pochhammer symbol
(or the rising factorial):
\[(\zeta)_n:=
\begin{cases}
1 & (n=0),\\
\zeta (\zeta +1)\cdots (\zeta +n-1) & (n=1,2,3,\ldots).
\end{cases}\]

\section{Differential equations for GJEFs}
\label{sec:differential}

In this section, we discuss ordinary differential equations satisfied by the GJEFs.

The function $\sn_{p,q,r}{x}$ is an increasing function on the interval $ [0, K_{p,q,r}(k)]$, as mentioned in Section \ref{sec:preparation}, 
and is extended to $\R$ as a $4K_{p,q,r}(k)$-periodic odd function by means of 
$\sn_{p,q,r} (-x)=-\sn_{p,q,r} x$ and $\sn_{p,q,r}(x+2K_{p,q,r}(k))=-\sn_{p,q,r}x$.
We will also call its extended function $\sn_{p,q,r}{x}$,
which is easily seen to be smooth in $\R$.
In the same manner, 
$\cn_{p,q,r}x$ and $\dn_{p,q,r}x$ are extended to $\R$
by $\cn_{p,q,r} (-x)=\cn_{p,q,r} x$, $\cn_{p,q,r}(x+2K_{p,q,r}(k))=-\cn_{p,q,r}x$ and 
$\dn_{p,q,r} (-x)=\dn_{p,q,r} x$, $\dn_{p,q,r}(x+2K_{p,q,r}(k))=\dn_{p,q,r}x$, respectively.
Then, $\cn_{p,q,r}x$ is a $4K_{p,q,r}(k)$-periodic even function
and $\dn_{p,q,r}x$ is a $2K_{p,q,r}(k)$-periodic even (positive) function.
For these extended functions, the properties given in the previous section can be easily extended by considering periodicity. 
For $\alpha>1$, let $\phi_\alpha (t):=|t|^{\alpha-2}t\ (t \in \R \setminus \{0\});\ =0\ (t=0)$.
Then, for example, \eqref{eq:cnp'} can be modified as
\[(\phi_p(\cn_{p,q,r}{x}))'
=-\frac{q}{p^*}\phi_q(\sn_{p,q,r}x)\dn_{p,q,r}{x}.
\]

\begin{thm}
\label{prop:snode}
$(\mathrm{i})$ For $x \in \R$, $y=\sn_{p,q,r}{x}$ satisfies 
\begin{equation}
\label{eq:snode}
(\phi_p(y'))'+\frac{q}{p^*}\phi_q(y)(1-k^q|y|^q)^{p/r^*-1}\left(1+\frac{p}{r^*}k^q-\left(1+\frac{p}{r^*}\right)k^q|y|^q\right)=0.
\end{equation}
In particular, $y=\sn_{p,q,p^*}x$ satisfies
the Allen-Cahn-type equation$:$
\begin{equation}
\label{eq:allencahn}
(\phi_p(y'))'+\frac{q}{p^*}\phi_q(y)(1+k^q-2k^q|y|^q)=0.
\end{equation}

$(\mathrm{ii})$ For $x \in \R$, $y=\phi_p(\cn_{p,q,r}{x})$ satisfies 
\begin{multline}
\label{eq:cnode}
(\phi_{q^*}(y'))'+\left(\frac{q}{p^*}\right)^{q^*-1}\phi_{p^*}(y)\\
\times (1-k^q+k^q|y|^{p^*})^{q^*/r^*-1}\left(1-\left(1+\frac{q^*}{r^*}\right)k^q+\left(1+\frac{q^*}{r^*}\right)k^q|y|^{p^*}\right)=0.
\end{multline}
In particular, $y=\phi_p(\cn_{p,q,q}{x})$ satisfies
\begin{equation}
\label{eq:cnallencahn}
(\phi_{q^*}(y'))'+\left(\frac{q}{p^*}\right)^{q^*-1}\phi_{p^*}(y)\left(1-2k^q+2k^q|y|^{p^*}\right)=0.
\end{equation}

$(\mathrm{iii})$ For $x \in \R$, $y=\phi_{r^*}(\dn_{p,q,r}{x})
=\dn_{p,q,r}^{r^*-1}{x}$ satisfies 
\begin{multline}
\label{eq:dnode}
(\phi_{q^*}(y'))'+\left(\frac{q}{r}\right)^{q^*-1}k^{q^*}y^{r-1}\\
\times \left(1-\frac{1}{k^q}+\frac{1}{k^q}y^r\right)^{q^*/p-1}\left(1-\left(1+\frac{q^*}{p}\right)\frac{1}{k^q}+\left(1+\frac{q^*}{p}\right)\frac{1}{k^q}y^r\right)=0.
\end{multline}
In particular, $y=\phi_{r^*}(\dn_{p,p^*,r}x)=\dn_{p,p^*,r}^{r^*-1}x$ satisfies the scalar-field-type equation$:$
\begin{equation}
\label{eq:dnallencahn}
(\phi_p(y'))'+\left(\frac{p^*}{r}\right)^{p-1}k^py^{r-1}
\left(1-\frac{2}{k^{p^*}}+\frac{2}{k^{p^*}}y^r\right)=0.
\end{equation}
\end{thm}

\begin{proof}
(i) Let $y=\sn_{p,q,r}{x}$. 
As $\phi_p(\cn_{p,q,r}x), \dn_{p,q,r}^{p-1}x \in C^1(\R)$,
we have $\phi_p(y') \in C^1(\R)$ and 
\begin{align*}
(\phi_p(y'))'
&=-\frac{q}{p^*}\phi_q(y)\dn_{p,q,r}^{p-r^*}x
\left(\dn_{p,q,r}^{r^*}x+\frac{p}{r^*}k^q|\cn_{p,q,r}x|^p\right)\\
&=-\frac{q}{p^*}\phi_q(y)(1-k^q|y|^q)^{p/r^*-1}
\left(1+\frac{p}{r^*}k^q-\left(1+\frac{p}{r^*}\right)k^q|y|^q\right),
\end{align*} 
i.e., $y$ satisfies \eqref{eq:snode}. When $r=p^*$, 
\eqref{eq:snode} becomes \eqref{eq:allencahn}.

(ii) Let $y=\phi_p(\cn_{p,q,r}{x})$. As $\sn_{p,q,r}x, \dn_{p,q,r}^{q^*-1}x \in C^1(\R)$,
we have $\phi_{q^*}(y') \in C^1(\R)$ and 
\begin{align*}
(\phi_{q^*}(y'))'
&=-\left(\frac{q}{p^*}\right)^{q^*-1}\cn_{p,q,r}x \dn_{p,q,r}^{q^*-r^*}x
\left(\dn_{p,q,r}^{r^*}x-\frac{q^*}{r^*}k^q|\sn_{p,q,r}x|^q\right)\\
&=-\left(\frac{q}{p^*}\right)^{q^*-1}\phi_{p^*}(y) (1-k^q+k^q|y|^{p^*})^{q^*/r^*-1}\\
&\qquad \qquad \times \left(1-\left(1+\frac{q^*}{r^*}\right)k^q+\left(1+\frac{q^*}{r^*}\right)k^q|y|^{p^*}\right),
\end{align*} 
i.e., $y$ satisfies \eqref{eq:cnode}. When $r=q$, 
\eqref{eq:cnode} becomes \eqref{eq:cnallencahn}.

(iii) Let $y=\phi_{r^*}(\dn_{p,q,r}{x})
=\dn_{p,q,r}^{r^*-1}{x}$. As $\sn_{p,q,r}x, \phi_{q^*}(\cn_{p,q,r}x) \in C^1(\R)$,
we have $\phi_{q^*}(y') \in C^1(\R)$ and 
\begin{align*}
(\phi_{q^*}(y'))'
&=-\left(\frac{q}{r}\right)^{q^*-1}k^{q^*}|\cn_{p,q,r}x|^{q^*-p} \dn_{p,q,r}x
\left(|\cn_{p,q,r}x|^p-\frac{q^*}{p}|\sn_{p,q,r}x|^q\right)\\
&=-\left(\frac{q}{r}\right)^{q^*-1}k^{q^*}y^{r-1} \left(1-\frac{1}{k^q}-\frac{1}{k^q}y^r\right)^{q^*/p-1}\\
&\qquad \qquad \times \left(1-\left(1+\frac{q^*}{p}\right)\frac{1}{k^q}+\left(1+\frac{q^*}{p}\right)\frac{1}{k^q}y^r\right),
\end{align*} 
i.e., $y$ satisfies \eqref{eq:dnode}. When $q=p^*$, 
\eqref{eq:dnode} becomes \eqref{eq:dnallencahn}.
\end{proof}

We conclude this section by giving some remarks on \eqref{eq:allencahn}. 
They are also given in the third author's paper \cite{Takeuchi2011}.

When $k=0$, \eqref{eq:allencahn} states that 
$\sn_{p,q,r}(x,0)=\sin_{p,q}x$ satisfies 
\begin{gather}
\label{eq:p/2ev}
(|y'|^{p-2}y')'+\frac{q}{p^*}|y|^{q-2}y=0.
\end{gather}
This fact is well known in the eigenvalue problem of the $p$-Laplacian, 
due to 
Dr\'{a}bek and Man\'{a}sevich \cite{DM1999}.

Let us assume that $p>1$, and let $k \to 1-0$
in \eqref{eq:allencahn} with $p$ replaced by $2p$. Then, we see that $y=\lim_{k \to 1-0}\sn_{2p,q,(2p)^*}(x,k)=\sin_{p,q}{x}$ satisfies
\begin{equation}
\label{eq:pallencahn}
(|y'|^{2p-2}y')'+\frac{2q}{(2p)^*}|y|^{q-2}y(1-|y|^q)=0.
\end{equation}
On the other hand, from \eqref{eq:ev}, $y=\sin_{p,q}x$ satisfies
\eqref{eq:p/2ev}.
This shows that an eigenfunction of $p$-Laplacian can be a solution to an Allen-Cahn-type equation with $2p$-Laplacian.
For any interval where $y' \neq 0$, this fact can be formally explained as follows. Since $y=\sin_{p,q}x$ is a $C^2$-function in this interval, 
$(|y'|^{2p-2}y')'=(2p-1)|y'|^{2p-2}y''$. 
Furthermore, $1-|y|^q=|y'|^{p}$. 
Therefore, dividing both sides of \eqref{eq:pallencahn} by $|y'|^{p}$ yields
\eqref{eq:p/2ev}.

\section{Integral formulae}
\label{sec:integral}

In this section, we give several integral formulae for GJEFs, including Wallis-type formulae.
As an application, the integral formulae for classical Jacobi elliptic functions are given.

\begin{thm}
\label{thm:intformula}
Let $p, q, r>1$ and $k \in [0,1)$. 
Then, for any $a>-1,\ b, c \in \R$ and $x \in [0,K_{p,q,r}(k))$,
\begin{multline}
\label{eq:primitive}
\int_0^x \sn_{p,q,r}^a{t}\cn_{p,q,r}^b{t}\dn_{p,q,r}^c{t}\,dt\\
=\frac{1}{a+1}\sn_{p,q,r}^{a+1}{x}
F_1\left(\frac{a+1}{q};\frac{1-b}{p},\frac{1-c}{r^*};\frac{a+1}{q}+1;
\sn_{p,q,r}^q{x},k^q\sn_{p,q,r}^q{x}\right).
\end{multline}
In particular, for any $a>-1,\ b>1-p,\ c \in \R$ and $k \in [0,1)$,
\begin{multline}
\label{eq:definite}
\int_0^{K_{p,q,r}(k)} 
\sn_{p,q,r}^a{t}\cn_{p,q,r}^b{t}\dn_{p,q,r}^c{t}\,dt\\
=\frac{1}{q}B\left(\frac{a+1}{q},\frac{b-1}{p}+1\right)
F\left(\frac{a+1}{q},\frac{1-c}{r^*};\frac{a+1}{q}+\frac{b-1}{p}+1;
k^q\right),
\end{multline}
where $F_1$ and $F$ are defined in Section \ref{sec:preparation}.
\end{thm}

\begin{proof}
Assume $a>-1,\ b, c \in \R$. Letting $\sn_{p,q,r}^q{t}=s$, we have
\begin{multline}
\label{eq:primitive2}
\int_0^x \sn_{p,q,r}^a{t}\cn_{p,q,r}^b{t}\dn_{p,q,r}^c{t}\,dt\\
=\frac{1}{q}\int_0^{\sn_{p,q,r}^q{x}}
s^{(a+1)/q-1}(1-s)^{(b-1)/p}(1-k^qs)^{(c-1)/r^*}\,ds.
\end{multline}
Here, it is known that for $\alpha>0,\ \beta, \gamma \in \R$ and $|x|, |y|<1$,
\begin{equation*}
\int_0^x s^{\alpha-1}(1-s)^\beta (1-ys)^{\gamma}\,ds
=\frac{x^\alpha}{\alpha}F_1(\alpha;-\beta,-\gamma;\alpha+1;x,yx)
\end{equation*}
(see \cite[5.8.2 (5)]{Erdelyi1953}). Hence, the right-hand side of  \eqref{eq:primitive2} is
\begin{equation*}
\frac{1}{a+1}\sn_{p,q,r}^{a+1}x
F_1\left(\frac{a+1}{q};\frac{1-b}{p},\frac{1-c}{r^*};
\frac{a+1}{q}+1;\sn_{p,q,r}^{q}x,k^q\sn_{p,q,r}^qx\right),
\end{equation*}
which implies \eqref{eq:primitive}.

Next, assume $a>-1,\ b>1-p,\ c \in \R$.
Letting $x \to K_{p,q,r}(k)-0$ in \eqref{eq:primitive}, we obtain
\begin{multline}
\label{eq:definite2}
\int_0^{K_{p,q,r}(k)} \sn_{p,q,r}^a{t}\cn_{p,q,r}^b{t}\dn_{p,q,r}^c{t}\,dt\\
=\frac{1}{a+1}
F_1\left(\frac{a+1}{q};\frac{1-b}{p},\frac{1-c}{r^*};\frac{a+1}{q}+1;
1,k^q\right).
\end{multline}
Here, it is also known that if $0<\alpha<\gamma,\ \gamma>\alpha+\beta$ and $|y|<1$, then
\[F_1(\alpha;\beta,\beta';\gamma;1,y)
=\frac{\Gamma(\gamma)\Gamma(\gamma-\alpha-\beta)}{\Gamma(\gamma-\alpha)\Gamma(\gamma-\beta)}F(\alpha,\beta';\gamma-\beta;y)\]
(see \cite[5.10 (10)]{Erdelyi1953}). Hence, the right-hand side of \eqref{eq:definite2} is 
\begin{equation*}
\frac{1}{q}B\left(\frac{a+1}{q},\frac{b-1}{p}+1\right)
F\left(\frac{a+1}{q},\frac{1-c}{r^*};\frac{a+1}{q}+\frac{b-1}{p}+1;
k^q\right),
\end{equation*}
which implies \eqref{eq:definite}.
\end{proof}

If $k=0$ in Theorem \ref{thm:intformula}, 
we obtain the integral formulae for GTFs.

\begin{cor}[{\cite[Theorem 3.1]{Kobayashi-Takeuchi}}]
\label{cor:sinacosb}
Let $p, q>1$. 
Then, for any $a>-1,\ b \in \R$ and $x \in [0,\pi_{p,q}/2)$,
\begin{equation*}
\int_0^x \sin_{p,q}^a{t}\cos_{p,q}^b{t}\,dt
=\frac{1}{a+1}\sin_{p,q}^{a+1}{x}
F\left(\frac{a+1}{q},\frac{1-b}{p};\frac{a+1}{q}+1;
\sin_{p,q}^q{x}\right).
\end{equation*}
In particular, for $a>-1,\ b>1-p$,
\begin{equation*}
\int_0^{\pi_{p,q}/2} 
\sin_{p,q}^a{t}\cos_{p,q}^b{t}\,dt
=\frac{1}{q}B\left(\frac{a+1}{q},\frac{b-1}{p}+1\right).
\end{equation*}
\end{cor}

\begin{proof}
Since $F_1(\alpha;\beta,\beta';\gamma;x,0)
=F(\alpha,\beta;\gamma;x)$, 
letting $k=0$ in \eqref{eq:primitive} and \eqref{eq:definite}
yields the assertions.
\end{proof}

The next theorem has already been proved in \cite{Takeuchi2016b}, but the proof is easy enough to state here. Note, however, that the definition of $K_{p,q,r}(k)$ in \cite{Takeuchi2016b} is slightly different from the definition in this present paper.

\begin{thm}[{\cite[p.41]{Takeuchi2016b}}]
\label{thm:choukika}
Let $p, q, r>1$ and $k \in [0,1)$. Then, 
\begin{align*}
& K_{p,q,r}(k)=\frac{\pi_{p,q}}{2}
F\left(\frac{1}{q},\frac{1}{r^*};\frac{1}{p^*}+\frac{1}{q};k^q\right),\\
& E_{p,q,r}(k)=\frac{\pi_{p,q}}{2}
F\left(\frac{1}{q},-\frac{1}{r};\frac{1}{p^*}+\frac{1}{q};k^q\right).
\end{align*}
\end{thm}

\begin{proof}
The formula for $K_{p,q,r}(k)$ is obtained immediately by setting $a=b=c=0$ in \eqref{eq:definite}. 
However, the following allows us to prove the formulae for $K_{p,q,r}(k)$ and 
$E_{p,q,r}(k)$ simultaneously.

Let $\alpha$ be any real number. Then, by the binomial series, 
\[\int_0^{\pi_{p,q}/2} (1-k^q\sin_{p,q}^q{\theta})^\alpha\,d\theta
=\sum_{n=0}^\infty (-1)^n \binom{\alpha}{n} k^{qn}
\int_0^{\pi_{p,q}/2} \sin_{p,q}^{qn}{\theta}\,d\theta.\]
Clearly, $(-1)^n \binom{\alpha}{n}=(-\alpha)_n/n!$.
Moreover, from \cite[Corollary 4]{Kobayashi-Takeuchi} (or Corollary 
\ref{cor:KT} below),
\[\int_0^{\pi_{p,q}/2} \sin_{p,q}^{qn}{\theta}\,d\theta
=\frac{(1/q)_n}{(1/p^*+1/q)_n} \frac{\pi_{p,q}}{2}.\] 
Hence, 
\begin{align*}
\int_0^{\pi_{p,q}/2} (1-k^q\sin_{p,q}^q{\theta})^\alpha\,d\theta
&=\frac{\pi_{p,q}}{2}
\sum_{n=0}^\infty \frac{(1/q)_n (-\alpha)_n}{(1/p^*+1/q)_n n!}(k^q)^n\\
&=\frac{\pi_{p,q}}{2}F\left(\frac{1}{q},-\alpha;\frac{1}{p^*}+\frac{1}{q};k^q\right).
\end{align*}
If $\alpha=-1/r^*$ and $1/r$ in the above equation, 
we obtain $K_{p,q,r}(k)$ and $E_{p,q,r}(k)$, respectively.
\end{proof}


By specializing \eqref{eq:definite} to a single function
in the integrand and writing the powers $a,\ b$ as multiples of 
$p$ and $q$ with remainders, we will give integral formulae that
generalize the famous Wallis integral formula \eqref{eq:WIF}. Note that for $\dn_{p,q,r}$, the result shows that there is no advantage in paying attention to the remainder of the power divided by $r$, so it is written as the power $c$.

In what follows, $1^*:=\infty,\ \pi_{p,1}:=2p^*,\ \pi_{\infty,q}:=2$.

\begin{thm}[Wallis-type integral formula I]
\label{thm:wallisI}
Let $p, q, r>1,\ k \in [0,1)$ and $n=0,1,2,\ldots$. 
Then, 
for any $R \in (-1,q-1]$,
\begin{equation*}
\int_0^{K_{p,q,r}(k)}\sn_{p,q,r}^{qn+R}(t,k)\,dt
=\frac{(1/u)_n}{(R+1)(1/p^*+1/u)_n} \frac{\pi_{p,u}}{2}
F\left(\frac{1}{u}+n,\frac{1}{r^*};\frac{1}{p^*}+\frac{1}{u}+n;k^q\right),
\end{equation*}
where $1/u:=(R+1)/q$$;$  for any $R \in (1-p,1]$,
\begin{equation*}
\int_0^{K_{p,q,r}(k)}\cn_{p,q,r}^{pn+R}(t,k)\,dt
=\frac{(1/v)_n}{(1/q+1/v)_n} \frac{\pi_{v^*,q}}{2}
F\left(\frac{1}{q},\frac{1}{r^*};\frac{1}{q}+\frac{1}{v}+n;k^q\right),
\end{equation*}
where $1/v:=(R+p-1)/p$$;$ for any $c \in \R$,
\begin{equation*}
\int_0^{K_{p,q,r}(k)}\dn_{p,q,r}^{c}(t,k)\,dt
=\frac{\pi_{p,q}}{2}
F\left(\frac{1}{q},\frac{1-c}{r^*};\frac{1}{p^*}+\frac{1}{q};k^q\right).
\end{equation*}
\end{thm}

\begin{proof}
Letting $a=qn+R,\ b=c=0$ in \eqref{eq:definite},
we have
\begin{equation*}
\int_0^{K_{p,q,r}(k)}\sn_{p,q,r}^{qn+R}{(t,k)}\,dt
=\frac{1}{q}B\left(n+\frac{1}{u},\frac{1}{p^*}\right)
F\left(n+\frac{1}{u},\frac{1}{r^*};n+\frac{1}{u}+\frac{1}{p^*};k^q
\right),
\end{equation*}
where $1/u:=(R+1)/q$. Since
\begin{equation*}
B\left(n+\frac{1}{u},\frac{1}{p^*}\right)
=\frac{(1/u)_n \Gamma(1/u)\Gamma(1/p^*)}{(1/u+1/p^*)_n\Gamma(1/u+1/p^*)}
=\frac{u(1/u)_n}{(1/p^*+1/u)_n}\frac{\pi_{p,u}}{2},
\end{equation*}
we obtain the formula of $\sn_{p,q,r}$.
The formulae of $\cn_{p,q,r},\ \dn_{p,q,r}$ are shown in a 
similar way.
\end{proof}

If $k=0$ in Theorem \ref{thm:wallisI}, 
we obtain the integral formulae for GTFs.

\begin{cor}[{\cite[Theorem 3.2]{Kobayashi-Takeuchi}}]
\label{cor:KT}
Let $p, q>1$ and $n=0,1,2,\ldots$. 
Then, 
for any $R \in (-1,q-1]$,
\begin{equation}
\label{eq:KT}
\int_0^{\pi_{p,q}/2}\sin_{p,q}^{qn+R}t\,dt
=\frac{(1/u)_n}{(R+1)(1/p^*+1/u)_n} \frac{\pi_{p,u}}{2},
\end{equation}
where $1/u:=(R+1)/q$$;$  for any $R \in (1-p,1]$,
\begin{equation*}
\int_0^{\pi_{p,q}/2}\cos_{p,q}^{pn+R}t\,dt
=\frac{(1/v)_n}{(1/q+1/v)_n} \frac{\pi_{v^*,q}}{2},
\end{equation*}
where $1/v:=(R+p-1)/p$.
\end{cor}

\begin{proof}
Letting $k=0$ in Theorem \ref{thm:wallisI}
yields the assertions.
\end{proof}

\begin{cor}
\label{cor:KTcor}
Let $q>1$ and $n=0,1,2,\ldots$. 
Then, 
\begin{align*}
& \int_0^{\pi_{q^*,q}/2}\sin_{q^*,q}^{qn}t\,dt
=\int_0^{\pi_{q^*,q}/2}\cos_{q^*,q}^{q^*n}t\,dt
=\frac{(1/q)_n}{(2/q)_n}\frac{\pi_{q^*,q}}{2},\\
& \int_0^{\pi_{q^*,q}/2}\sin_{q^*,q}^{qn+q-1}t\,dt
=\int_0^{\pi_{q^*,q}/2}\cos_{q^*,q}^{q^*n+1}t\,dt
=\frac{(1)_n}{(1/q+1)_n}.
\end{align*}
In particular, letting $q=2$ yields the Wallis integral formula \eqref{eq:WIF}.
\end{cor}

\begin{proof}
They are easily obtained by setting $p=q^*$ and $R$ appropriately
in \eqref{eq:KT} of Corollary \ref{cor:KT}. 
When $q=2$, \eqref{eq:WIF} follows immediately from these formulae above if we note that $(1/2)_n/(1)_n=(2n-1)!!/(2n)!!$ and
$(1)_n/(3/2)_n=(2n)!!/(2n+1)!!$.
\end{proof}

\begin{rem}
From \cite[Appendix]{Kobayashi-Takeuchi}, we know that
\[\sin_{q^*,q}t=\cos_{q^*,q}^{q^*-1}\left(\frac{\pi_{q^*,q}}{2}-t\right) \quad (t \in [0,\pi_{q^*,q}/2]),\]
and the first equality in each of the formulae in Corollary \ref{cor:KTcor}
can also be proved from that fact.
\end{rem}

Next, we propose another Wallis-type integral formula that is different from Theorem \ref{thm:wallisI}. In preparation for this, we will construct a recurrence relation that is useful in its own right.
Let
\begin{equation}
I_a=I_a(k):=\int_0^{K_{p,q,r}(k)} \sn_{p,q,r}^a{(t,k)}\,dt.
\end{equation}
$I_a(k)$  satisfies the following recurrence relation.

\begin{thm}
\label{thm:Iarelation}
Let $p, q, r>1$ and $k \in [0,1)$. Then, for any $a>-1$,
\begin{equation}
\label{eq:Iarelation}
\left(\frac{1}{p^*}+\frac{a+1}{q}+\frac{1}{r}\right)k^qI_{a+2q}
-\left(\frac{1}{p^*}+\frac{a+1}{q}+\left(\frac{a+1}{q}+\frac{1}{r}\right)k^q\right)I_{a+q}
+\frac{a+1}{q}I_a=0.
\end{equation}
\end{thm}

\begin{proof}
From \eqref{eq:cs}--\eqref{eq:dnp'}, we have
\begin{multline*}
(\sn_{p,q,r}^{a+1}{x}\cn_{p,q,r}^{p-1}{x}\dn_{p,q,r}^{r^*-1}{x})'\\
=(a+1)\sn_{p,q,r}^{a}{x}
-\left(a+1+\frac{q}{p^*}+\left(a+1+\frac{q}{r}\right)k^q\right)\sn_{p,q,r}^{a+q}{x}\\
+\left(a+1+\frac{q}{p^*}+\frac{q}{r}\right)k^q\sn_{p,q,r}^{a+2q}{x}.
\end{multline*}
Integrating both sides on $[0,K_{p,q,r}(k)]$, we obtain \eqref{eq:Iarelation}.
\end{proof}

In particular, in the recurrence relation \eqref{eq:Iarelation}, 
setting $k=0$ yields the relation for the case of GTFs, and setting $p=q=r=2$ yields the relation for the case of classical Jacobi elliptic functions.

\begin{cor}[{\cite{Kobayashi-Takeuchi}, \cite[310.05, 310.06]{Byrd-Friedman1971}}]
$I_a=\int_0^{\pi_{p,q}/2}\sin_{p,q}^a{t}\,dt$ satisfies
\begin{equation*}
I_{a+q}=\frac{a+1}{a+1+q/p^*}I_{a}.
\end{equation*}
$I_a=\int_0^{K(k)} \sn^a{(t,k)}\,dt$ satisfies
\begin{equation*}
(a+3)k^2I_{a+4}-(a+2)(1+k^2)I_{a+2}+(a+1)I_a=0.
\end{equation*}
\end{cor}


We prepare symbols before stating another Wallis-type integral formula.
We had defined $K_{p,q,r}(k)$ and $E_{p,q,r}(k)$ for $p, q, r>1$, 
but we extend them to $q=1$ by the equations \eqref{eq:kpqr} and \eqref{eq:epqr} 
with $q=1$ in their respective second integral expressions from the last
(i.e., expression without trigonometric functions). That is,
for $p, r>1$,
\begin{align*}
& K_{p,1,r}(k):=\int_0^1 \frac{dt}{(1-t)^{1/p}(1-kt)^{1-1/r}},\\
& E_{p,1,r}(k):=\int_0^1 \frac{(1-kt)^{1/r}}{(1-t)^{1/p}}\,dt.
\end{align*}

\begin{thm}[Wallis-type integral formula II]
\label{thm:wallisII}
Let $p, q, r>1,\ k \in (0,1)$ and $n=0,1,2,\ldots$. 
Then, 
for any $R \in (-1,q-1]$,
\begin{multline*}
\int_0^{K_{p,q,r}(k)}\sn_{p,q,r}^{qn+R}(t,k)\,dt
=\frac{1}{(R+1)(1/p^*+1/u+1/r)_n}\\
\times
\frac{(w_{2,1}^{(n)}+k^qw_{2,2}^{(n)})K_{p,u,r}(k^{R+1})
-w_{2,1}^{(n)}E_{p,u,r}(k^{R+1})}{k^{q(n+1)}},
\end{multline*}
where $1/u:=(R+1)/q$, and 
the matrix $W_n:=(w_{i,j}^{(n)})_{1 \le i,j \le 2}$ is
\begin{equation}
\label{eq:W_n}
W_n:=\begin{cases}
E & (n=0),\\
A_{n-1}A_{n-2}\cdots A_2A_1A_0 & (n=1,2,3,\ldots)
\end{cases}
\end{equation}
with the identity matrix $E$ and
\begin{equation}
\label{eq:A_n}
A_n:=
\begin{pmatrix}
n+1/u+1/p^*+(n+1/u+1/r)k^q 
& -1/u-n\\
(n+1/u+1/p^*+1/r)k^q
&0
\end{pmatrix}.
\end{equation}
\end{thm}

\begin{proof}
From Theorem \ref{thm:Iarelation}, 
\begin{equation*}
I_{a+2q}=\frac{1}{(a+1+q/p^*+q/r)k^q}
\left(\left(a+1+\frac{q}{p^*}+\left(a+1+\frac{q}{r}\right)k^q\right)I_{a+q}-(a+1)I_a\right).
\end{equation*}
Let $a=qn+R$, where $n=0,1,2,\ldots$, $R \in (-1,q-1]$, and
\begin{equation*}
\mathbb{I}_{n}
=\begin{pmatrix}
I_{q(n+1)+R} \\ I_{qn+R}
\end{pmatrix}.
\end{equation*}
Then, we see that
\begin{equation*}
\mathbb{I}_{n+1}=\frac{1}{(n+1/u+1/p^*+1/r)k^q}A_n \mathbb{I}_n,
\end{equation*}
where $A_n$ is the matrix defined as \eqref{eq:A_n}.
It follows from the recurrence formula that
\[\mathbb{I}_n
=\frac{1}{(1/u+1/p^*+1/r)_nk^{qn}}W_n\mathbb{I}_0,\]
that is,
\begin{equation}
\label{eq:Iqn+R}
I_{qn+R}
=\frac{1}{(1/u+1/p^*+1/r)_nk^{qn}}
(w_{2,1}^{(n)}I_{q+R}+w_{2,2}^{(n)}I_R),
\end{equation}
where $W_n=(w_{i,j}^{(n)})_{1\leq i,j \leq 2}$ is the matrix defined as
\eqref{eq:W_n}.
Here, Theorems \ref{thm:wallisI} and \ref{thm:choukika} yield
\begin{align*}
& I_R=\frac{1}{R+1}K_{p,u,r}(k^{R+1}),\\
& I_{q+R}
=\frac{1}{q(1/p^*+1/u)}\frac{\pi_{p,u}}{2}
F\left(\frac{1}{u}+1,\frac{1}{r^*};\frac{1}{p^*}+\frac{1}{u}+1;k^q\right).
\end{align*}
To compute $I_{q+R}$, we use
\[F(\alpha+1,\beta+1;\gamma+1;x)
=\frac{\gamma}{\alpha x}
(F(\alpha,\beta+1;\gamma;x)-F(\alpha,\beta;\gamma;x)),\]
which can be proved by straight computation from the definition of hypergeometric series.
Then, Theorem \ref{thm:choukika} yields
\begin{align*}
I_{q+R}
&=\frac{1}{(R+1)k^q}(K_{p,u,r}(k^{R+1})-E_{p,u,r}(k^{R+1})).
\end{align*}
Therefore, by \eqref{eq:Iqn+R},
\begin{multline*}
I_{qn+R}
=\frac{1}{(R+1)(1/p^*+1/u+1/r)_n}\\
\times
\frac{(w_{2,1}^{(n)}+k^qw_{2,2}^{(n)})K_{p,u,r}(k^{R+1})
-w_{2,1}^{(n)}E_{p,u,r}(k^{R+1})}{k^{q(n+1)}}.
\end{multline*}
\end{proof}

\begin{rem}
The formulae for $\int \cn_{p,q,r}^bt\,dt$ and $\int \dn_{p,q,r}^ct\,dt$ can be deduced
in the same way and we leave the proofs for readers.
In order to derive the recurrence formulae as given in Theorem \ref{thm:Iarelation}, 
we can differentiate $\sn_{p,q,r}x \cn_{p,q,r}^{b+p-1}x\dn_{p,q,r}^{r^*-1}x$
for $\int \cn_{p,q,r}^bt\,dt$ and 
$\sn_{p,q,r}x \cn_{p,q,r}^{p-1}x\dn_{p,q,r}^{c+r^*-1}x$ for 
$\int \dn_{p,q,r}^ct\,dt$.
\end{rem}

In the following, Theorem \ref{thm:wallisII} is used to give the integral formulae for the powers of the classical Jacobi elliptic functions. Naturally, these can also be obtained using the recurrence relation \eqref{eq:Izenka}.

\begin{exmp}
Putting $p=q=r=2$ and $R=0$ in Theorem \ref{thm:wallisII}, we obtain
\[I_{2n}=\int_0^{K(k)}\sn^{2n}(t,k)\,dt
=\frac{(w_{2,1}^{(n)}+k^2w_{2,2}^{(n)})K(k)
-w_{2,1}^{(n)}E(k)}{(\frac32)_nk^{2(n+1)}},\]
and
\[A_n=
\begin{pmatrix}
(1+n)(1+k^2) & -\frac12-n \\
(\frac32+n)k^2 & 0
\end{pmatrix}.
\]

For $n=0$, since $W_0=E$, we see that $w_{2,1}^{(0)}=0$ and $w_{2,2}^{(0)}=1$.
Therefore, $I_0=K(k)$.

For $n=1$,
\[W_1=A_0=
\begin{pmatrix}
1+k^2 & -\frac12 \\
w_{2,1}^{(1)} & w_{2,2}^{(1)}
\end{pmatrix},\]
where $w_{2,1}^{(1)}=\frac32k^2$ and $w_{2,2}^{(1)}=0$.
Therefore,
\[I_2=\frac{K(k)-E(k)}{k^2}.\]

For $n=2$, 
\[W_2=A_1A_0=
\begin{pmatrix}
2+\frac74k^2+2k^4 & -1-k^2 \\
w_{2,1}^{(2)} & w_{2,2}^{(2)}
\end{pmatrix},\]
where $w_{2,1}^{(2)}=\frac52k^2+\frac52k^4$ and $w_{2,2}^{(2)}=-\frac54k^2$.
Therefore,
\[I_4=\frac{(2+k^2)K(k)-2(1+k^2)E(k)}{3k^4}.\]

For $n=3$,
\[W_3=A_2A_1A_0=
\begin{pmatrix}
(1+k^2)(6-k^2+6k^4) & -3-\frac{23}{8}k^2-3k^4 \\
w_{2,1}^{(3)} & w_{2,2}^{(3)}
\end{pmatrix},\]
where $w_{2,1}^{(3)}=k^2(7+\frac{49}{8}k^2+7k^4)$ and $w_{2,2}^{(3)}=-\frac72k^2(1+k^2)$.
Therefore,
\[I_6=\frac{(8+3k^2+4k^4)K(k)-(8+7k^2+8k^4)E(k)}{15k^6}.\]
\end{exmp}

\begin{exmp}
Putting $p=q=r=2$ and $R=1$ in Theorem \ref{thm:wallisII}, we obtain
\[I_{2n+1}=\int_0^{K(k)}\sn^{2n+1}(t,k)\,dt
=\frac{(w_{2,1}^{(n)}+k^2w_{2,2}^{(n)})K_{2,1,2}(k^2)
-w_{2,1}^{(n)}E_{2,1,2}(k^2)}{2(2)_nk^{2(n+1)}},\]
and
\[A_n=
\begin{pmatrix}
(\frac32+n)(1+k^2) & -1-n \\
(2+n)k^2 & 0
\end{pmatrix}.
\]
Here, $K_{2,1,2}(k)$ and $E_{2,1,2}(k)$ are 
\begin{align*}
K_{2,1,2}(k)
&=\int_0^1 \frac{dt}{\sqrt{(1-t)(1-kt)}}
=\frac{1}{\sqrt{k}}\log{\frac{1+\sqrt{k}}{1-\sqrt{k}}},\\
E_{2,1,2}(k)
&=\int_0^1 \sqrt{\frac{1-kt}{1-t}}\,dt
=1+\frac{1-k}{2\sqrt{k}}\log{\frac{1+\sqrt{k}}{1-\sqrt{k}}};
\end{align*}
hence, 
\[I_{2n+1}
=\frac{1}{2(2)_nk^{2(n+1)}}\left(-w_{2,1}^{(n)}
+\left(\frac{1+k^2}{2k}w_{2,1}^{(n)}+kw_{2,2}^{(n)}\right)
\log{\frac{1+k}{1-k}}\right).\]

For $n=0$, since $W_0=E$, we see that $w_{2,1}^{(0)}=0$ and $w_{2,2}^{(0)}=1$. Therefore,
\[I_1=\frac{1}{2k}\log{\frac{1+k}{1-k}}.\]

For $n=1$,
\[W_1=A_0=
\begin{pmatrix}
\frac32(1+k^2) & -1 \\
w_{2,1}^{(1)} & w_{2,2}^{(1)}
\end{pmatrix},\]
where $w_{2,1}^{(1)}=2k^2$ and $w_{2,2}^{(1)}=0$.
Therefore,
\[I_3=-\frac{1}{2k^2}+\frac{1+k^2}{4k^3}\log{\frac{1+k}{1-k}}.\]

For $n=2$, 
\[W_2=A_1A_0=
\begin{pmatrix}
\frac{15}{4}(1+k^2)^2-4k^2 & -\frac52(1+k^2)\\
w_{2,1}^{(2)} & w_{2,2}^{(2)}
\end{pmatrix},\]
where $w_{2,1}^{(2)}=\frac92k^2(1+k^2)$ and $w_{2,2}^{(2)}=-3k^2$.
Therefore,
\[I_5=-\frac{3(1+k^2)}{8k^4}+\frac{3k^4+2k^2+3}{16k^5}
\log{\frac{1+k}{1-k}}.\]
\end{exmp}

\section{Legendre-type relation}
\label{sec:legendre}

In this section, we prove a Legendre-type relation that holds for GCEIs with three parameters. In fact, this relation formula has already been obtained in the third author's paper \cite{Takeuchi2016b}, and it turns out to be equivalent to Elliott's identity in hypergeometric series theory. In the present paper, we provide another proof to clarify the meaning of this relation formula, which was unclear in \cite{Takeuchi2016b}. A corresponding consideration in the theory of hypergeometric series is given by Anderson, Vamanamurthy and Vuorinen \cite{AVV2001}.

Throughout this section, we will denote 
\[K_{p,q,r}'(k):=K_{p,q,r}(k_{q,r}'), \quad 
E_{p,q,r}'(k):=E_{p,q,r}(k_{q,r}'),\]
where $k'_{q,r}:=(1-k^r)^{1/q}$,
according to the traditional notation. 
In the following, these functions with $q$ and $r$ interchanged appear exclusively. The expressions are specifically written as follows:
\begin{align*}
K_{p,r,q}'(k)&=K_{p,r,q}(k_{r,q}')=\int_0^{\pi_{p,r}/2}
\frac{d\theta}{(1-(1-k^q)\sin_{p,r}^r\theta)^{1-1/q}},\\
E_{p,r,q}'(k)&=E_{p,r,q}(k_{r,q}')=\int_0^{\pi_{p,r}/2}
(1-(1-k^q)\sin_{p,r}^r\theta)^{1/q}\,d\theta.
\end{align*}
From now on, unless otherwise noted, 
\[k':=k_{r,q}'=(1-k^q)^{1/r}.\]

Moreover, the constant
\[\alpha:=\frac{1}{q}+\frac{1}{r}-\frac{1}{p}\]
is important.

\begin{lem}
\label{lem:diff}
\begin{align}
\label{eq:dK}
\frac{dK_{p,q,r}(k)}{dk}&=\dfrac{-(\alpha q-k^q)K_{p,q,r}(k)+\alpha q E_{p,q,r}(k)}{k(k')^r},\\
\label{eq:dE}
\frac{dE_{p,q,r}(k)}{dk}&=\frac{q(-K_{p,q,r}(k)+E_{p,q,r}(k))}{rk},\\
\label{eq:dK'}
\frac{dK_{p,r,q}'(k)}{dk}&=\frac{q((\alpha r-(k')^r)K_{p,r,q}'(k)-\alpha r E_{p,r,q}'(k))}{rk(k')^r},\\
\label{eq:dE'}
\frac{dE_{p,r,q}'(k)}{dk}&=\frac{k^{q-1}(K_{p,r,q}'(k)-E_{p,r,q}'(k))}{(k')^r}.
\end{align}
\end{lem}

\begin{proof}
See \cite[Proposition 2]{Takeuchi2016b} for \eqref{eq:dK} and \eqref{eq:dE}. 
Note, however, that the definition of $K_{p,q,r}$ in \cite{Takeuchi2016b} is slightly different
from that of ours, but the same
formula holds true.
Eqs. \eqref{eq:dK'} and \eqref{eq:dE'} are obtained from the derivative of the composite function with $d(k')/dk=-qk^{q-1}/(r(k')^{r-1})$.
\end{proof}

Let $y_1,\ y_2$ be the functions: for $k \in (0,1)$
\begin{align*}
& y_1(k):=K_{p,q,r}(k),\\
& y_2(k):=k^{q/p-1}(k')^{1-r/p}K_{p,r,q}'(k).
\end{align*}

\begin{lem}
Functions $y_1,\ y_2$ are solutions of
\begin{equation}
\label{eq:y1y2sol}
k(k')^r\frac{d^2y}{dk^2}+\left(2-\frac{q}{p}-\left(2+\frac{q}{r^*}\right)k^q\right)\frac{dy}{dk}-\frac{q}{r^*}k^{q-1}y=0.
\end{equation}
\end{lem}

\begin{proof}
First, let us discuss $y_1$. From Theorem \ref{thm:choukika},
$y_1$ is represented by using the hypergeometric series  $y=F(1/q,1/r^*;1/p^*+1/q;x)$, which satisfies
the hypergeometric differential equation:
\[x(1-x)\frac{d^2y}{dx^2}+\left(\frac{1}{p^*}+\frac{1}{q}
-\left(\frac{1}{q}+\frac{1}{r^*}+1\right)x\right)\frac{dy}{dx}-\frac{1}{qr^*}y=0.\]
Therefore, letting $x=k^q$ yields
\begin{equation}
\label{eq:y1sol}
k(k')^r\frac{d^2y_1}{dk^2}+\left(2-\frac{q}{p}-\left(2+\frac{q}{r^*}\right)k^q\right)\frac{dy_1}{dk}-\frac{q}{r^*}k^{q-1}y_1=0,
\end{equation}
which implies that $y_1$ satisfies \eqref{eq:y1y2sol}.

Next, consider $y_2$. Put $z:=\phi(k)y_2=K_{p,r,q}'(k)$, where
\begin{equation}
\label{eq:phi}
\phi(k):=k^{1-q/p}(k')^{r/p-1}.
\end{equation}
The hypergeometric series $y=F(1/r,1/q^*;1/p^*+1/r;1-x)$ satisfies the hypergeometric differential equation:
\[x(1-x)\frac{d^2y}{dx^2}+\left(\frac{1}{p}+\frac{1}{q^*}
-\left(\frac{1}{q^*}+\frac{1}{r}+1\right)x\right)\frac{dy}{dx}-\frac{1}{q^*r}y=0.\]
Similar to the previous calculation of $y_1$, 
\[k(k')^r\frac{d^2z}{dk^2}+\left(\frac{q}{p}-q\left(\frac{1}{r}+1\right)k^q\right)\frac{dz}{dk}-\frac{q(q-1)}{r}k^{q-1}z=0.\]
Dividing both sides by $\phi(k)$, we have
\[\frac{d}{dk}\left(k^{q/p}(k')^{r/p^*+1}\frac{dz}{dk}\right)
=\frac{q(q-1)}{r}k^{q-1}\frac{z}{\phi(k)},\]
i.e.,
\[\frac{d}{dk}\left(k(k')^r\frac{dy_2}{dk}
+\left(1-\frac{q}{p}-\left(1-\frac{q}{r}\right)k^q\right)y_2\right)
=\frac{q(q-1)}{r}k^{q-1}y_2.\]
After computing the left-hand side and then rearranging, we conclude
\begin{equation}
\label{eq:y2sol}
k(k')^r\frac{d^2y_2}{dk^2}+\left(2-\frac{q}{p}-\left(2+\frac{q}{r^*}\right)k^q\right)\frac{dy_2}{dk}-\frac{q}{r^*}k^{q-1}y_2=0,
\end{equation}
which implies that $y_2$ also satisfies \eqref{eq:y1y2sol} as well as $y_1$.
\end{proof}

Define the Wronskian of $y_1,\ y_2$ as
\[W(y_1,y_2;k):=y_1\frac{dy_2}{dk}-\frac{dy_1}{dk}y_2.\]

\begin{lem}
\label{lem:WisC}
There exists a constant $C$ such that for any $k \in (0,1)$,
\begin{equation}
\label{eq:WisC}
W(y_1,y_2;k)=\frac{C}{k^{2-q/p}(k')^{r/p+r-1}}.
\end{equation}
\end{lem}

\begin{proof}
Subtracting \eqref{eq:y1sol} multiplied by $y_2$ from \eqref{eq:y2sol} multiplied by $y_1$ yields
\[k(k')^r\frac{dW}{dk}+\left(2-\frac{q}{p}-\left(2+\frac{q}{r^*}\right)k^q\right)W=0.\]
Multiplying both sides by $\phi(k)$ leads to
\[\frac{d}{dk}\left(k^{2-q/p}(k')^{r/p+r-1}W\right)=0.\]
Hence, $k^{2-q/p}(k')^{r/p+r-1}W \equiv C$ for some 
constant $C$.
\end{proof}

Now, we define
\[L(k):=E_{p,q,r}(k)K_{p,r,q}'(k)
+K_{p,q,r}(k)E_{p,r,q}'(k)
-K_{p,q,r}(k)K_{p,r,q}'(k).
\]
\begin{lem}
\label{lem:WisL}
For any $k \in (0,1)$,
\begin{equation}
\label{eq:WisL}
W(y_1,y_2;k)=-\frac{\alpha qL(k)}{k^{2-q/p}(k')^{r/p+r-1}}.
\end{equation}
\end{lem}

\begin{proof}
From \eqref{eq:dK} in Lemma \ref{lem:diff},
\begin{equation}
\label{eq:diffy1}
k(k')^r\frac{dy_1}{dk}=-(\alpha q-k^q)K_{p,q,r}(k)+\alpha qE_{p,q,r}(k).
\end{equation}

On the other hand, 
differentiating both sides of $\phi (k)y_2=K_{p,r,q}'(k)$, where $\phi(k)$ is defined as \eqref{eq:phi},
and multiplying them by $k(k')^r$ give
\begin{multline*}
\left(1-\frac{q}{p}-\left(1-\frac{q}{r}\right)k^q\right)K_{p,r,q}'(k)
+k(k')^r\phi(k)\frac{dy_2}{dk}\\
=\frac{q}{r}((\alpha r-(k')^r)K_{p,r,q}'(k)-\alpha rE_{p,r,q}'(k)).
\end{multline*}
Here, \eqref{eq:dK'} in Lemma \ref{lem:diff} 
was used in the calculation of the right-hand side.
Thus,
\begin{equation}
\label{eq:diffy2}
k(k')^r\frac{dy_2}{dk}
=\frac{k^qK_{p,r,q}'(k)-\alpha qE_{p,r,q}'(k)}{\phi(k)}.
\end{equation}
Subtracting \eqref{eq:diffy1} multiplied by $y_2$ from \eqref{eq:diffy2} multiplied by $y_1$ yields
\[W=-\frac{\alpha qL(k)}{k(k')^r\phi(k)}.\]
This proof is complete.
\end{proof}

\begin{thm}[Three parameters]
\label{thm:3para}
Let $p, q, r>1$.
Then,  for any $k \in (0,1)$, $L(k)$ is non-zero constant, precisely,
\begin{multline}
\label{eq:LR3para}
E_{p,q,r}(k)K_{p,r,q}'(k)
+K_{p,q,r}(k)E_{p,r,q}'(k)
-K_{p,q,r}(k)K_{p,r,q}'(k)\\
=\frac{\pi_{p,q}}{2r}B\left(\frac{1}{p^*}+\frac{1}{q},\frac{1}{r}\right).
\end{multline}
Hence, it follows from Lemma \ref{lem:WisL} that $y_1$ and $y_2$ are linearly dependent if and only if $\alpha = 0$, i.e., $1/q+1/r=1/p$.
\end{thm}

\begin{proof}
Combining Lemmas  \ref{lem:WisC} and \ref{lem:WisL}, 
we see that $\alpha L(k)$ is constant.

First, we consider the case $\alpha \neq 0$. In this case, 
$L(k)$ is a constant $C$.  In the same way as \cite[Theorem 1]{Takeuchi2016b},
we can show 
\[|(E_{p,q,r}(k)-K_{p,q,r}(k))K_{p,r,q}'(k)|
\le \frac{\pi_{p,r}}{2}kK_{p,q,r}(k),\]
hence $(E_{p,q,r}(k)-K_{p,q,r}(k))K_{p,r,q}'(k) \to 0$ as $k \to +0$. 
Moreover, by Corollary \ref{cor:sinacosb},
\[C=\lim_{k \to +0}L(k)=\lim_{k \to +0}K_{p,q,r}(k)E_{p,r,q}'(k)=\frac{\pi_{p,q}}{2}E_{p,r,q}(1)
=\frac{\pi_{p,q}}{2r}B\left(\frac{1}{p^*}+\frac{1}{q},\frac{1}{r}\right).\]

Next, we consider the case $\alpha=0$.
Then, solving \eqref{eq:dK} and \eqref{eq:dK'} in Lemma \ref{lem:diff}, we obtain $K_{p,q,r}(k)=(\pi_{p,q}/2)(k')^{-r/q}$ and  $K_{p,r,q}'(k)=(\pi_{p,r}/2)k^{-q/r}$, respectively. Therefore,
\[L(k)=\frac{\pi_{p,r}}{2}k^{-q/r}E_{p,q,r}(k)
+\frac{\pi_{p,q}}{2}(k')^{-r/q}E_{p,r,q}'(k)
-\frac{\pi_{p,q}\pi_{p,r}}{4}
k^{-q/r}(k')^{-r/q}.\]
Here, \eqref{eq:dE} in Lemma \ref{lem:diff} again gives
\begin{gather*}
\frac{dE_{p,q,r}(k)}{dk}
=\frac{q}{rk}E_{p,q,r}(k)-\frac{q\pi_{p,q}}{2rk}(k')^{-r/q};
\end{gather*}
hence, 
\[\frac{d}{dk}(k^{-q/r}E_{p,q,r}(k))=-\frac{q\pi_{p,q}}{2r}k^{-q/r-1}(k')^{-r/q}.\]
Similarly, from \eqref{eq:dE'}, we have
\[\frac{d}{dk}((k')^{-r/q}E_{p,r,q}'(k))=\frac{\pi_{p,r}}{2}k^{q/r^*-1}(k')^{-r/q-r}.\]
Thus,
\begin{align*}
\frac{dL(k)}{dk}
&=\frac{\pi_{p,q}\pi_{p,r}}{4}
\left(-\frac{q}{r}k^{-q/r-1}(k')^{-r/q}
+k^{q/r^*-1}(k')^{-r/q-r} \right.\\
& \hspace{3cm}\left. +\frac{q}{r}k^{-q/r-1}(k')^{-r/q}
-k^{q/r^*-1}(k')^{-r/q-r}\right)=0.
\end{align*}
This yields that $L(k)$ is a constant,
which is obtained in the same way as for $\alpha \neq 0$.

Finally, since $L(k)$ is non-zero constant, it follows from Lemma \ref{lem:WisL} that $y_1$ and $y_2$ are linearly dependent if and only if $\alpha = 0$, i.e., $1/q+1/r=1/p$.
\end{proof}

\begin{rem}
If we only need to prove this identity, we can prove it without making a case for $\alpha$ by directly differentiating the defining formula of $L(k)$ and applying Lemma \ref{lem:diff} (see \cite[Theorem 1]{Takeuchi2016b}). On the other hand, the above proof is superior in that it implies that $L(k)$ is a Wronskian of two functions, as stated in Lemma \ref{lem:WisL}.
\end{rem}

Define $K_{p,q}'(k):=K_{p,q}(k_{q,q}'),\ E_{p,q}'(k):=E_{p,q}(k_{q,q}')$, where $k_{q,q}'=(1-k^q)^{1/q}$.

\begin{cor}[Two parameters]
Let $p, q>1$.
Then, for any $k \in (0,1)$,
\begin{multline}
E_{p,q}(k)K_{p,q}'(k)
+K_{p,q}(k)E_{p,q}'(k)
-K_{p,q}(k)K_{p,q}'(k)
=\frac{\pi_{p,q}}{2q}B\left(\frac{1}{p^*}+\frac{1}{q},\frac{1}{q}\right).
\end{multline}
\end{cor}

Define $K_{p}'(k):=K_p(k_{p,p}'),\ E_{p}'(k):=E_p(k_{p,p}')$, where
$k_{p,p}'=(1-k^p)^{1/p}$.

\begin{cor}[One parameter]
Let $p>1$. Then, for any $k \in (0,1)$,
\begin{equation}
\label{eq:LR1para}
E_{p}(k)K_{p}'(k)
+K_{p}(k)E_{p}'(k)
-K_{p}(k)K_{p}'(k)
=\frac{\pi_{p}}{2}.
\end{equation}
In particular, when $p=2$, \eqref{eq:LR1para} is identical to the classical
Legendre relation \eqref{eq:CLR}.
\end{cor}

\begin{rem}
Using the functions
\begin{align*}
z_1(k)
&:=k^{1-q/(2p)}(k')^{(r/2)(1/p+1/r^*)}K_{p,q,r}(k),\\
z_2(k)
&:=k^{q/(2p)}(k')^{(r/2)(1/p^*+1/r)}K'_{p,r,q}(k),
\end{align*}
instead of $y_1, y_2$, we can eliminate the denominators on the right-hand side of \eqref{eq:WisC} and \eqref{eq:WisL}, respectively.
Namely, we obtain
\begin{align*}
&W(z_1,z_2;k)=C, \quad W(z_1,z_2;k)=-\alpha L(k).
\end{align*}
From these two expressions, it follows that $\alpha L(k)$ is a constant, and the proof of Theorem \ref{thm:3para} is directly applicable.
Thus, the Legendre-type relation \eqref{eq:LR3para} can be clearly derived from the fact that the Wronskian of 
$z_1$ and $z_2$ is just a constant.
In particular, if $p=q=r=2$, then $z_1(k)=\sqrt{k}k'K(k)$ and $z_2(k)=\sqrt{k}k'K'(k)$. These functions correspond to those used by Borwein and Borwein \cite[Theorem 1.6]{BB1998} in their proof of the Legendre relation.
However, we adopt here the proof using $y_1$ and $y_2$, where at least one of the functions is clearly $K_{p,q,r}(k)$.
\end{rem}

\section{Inequalities on binomial expansions}
\label{sec:edmunds}

Edmunds and Lang \cite{Edmunds2023} gave the following inequality for GTFs with one parameter.

\begin{thm}[\cite{Edmunds2023}]
\label{thm:elineq}
Let $p \geq 2$ and $x \in [0,\pi_p/2]$. 
Then,
\begin{equation*}
(\sin_px+\cos_px)^{p^*}
\leq 1+2^{2/p}\sin_px\cos_px.
\end{equation*}
This inequality is reversed if $p \in (1,2]$.
\end{thm}

The proof of Theorem \ref{thm:elineq} essentially owes to the following nontrivial inequality by Carlen et al. \cite{Carlen2021}:
\begin{equation}
\label{eq:liebineq}
1 \leq \left(1+\left(\frac{2\alpha^{p/2}(1-\alpha)^{p/2}}{\alpha^p+(1-\alpha)^p}\right)^{2/p}\right)^{p-1}(\alpha^p+(1-\alpha)^p)
\quad \mbox{for all $\alpha \in [0,1]$}
\end{equation}
when $p \in (0,1] \cup [2,\infty)$, with the reverse inequality when
$p \in (-\infty,0) \cup (1,2]$.
Theorem \ref{thm:elineq} above is obtained immediately by putting
\[\alpha=\frac{\sin_px}{\sin_px+\cos_px}\]
in \eqref{eq:liebineq}.

We can extend Theorem \ref{thm:elineq} for GJEFs with three parameters as follows.

\begin{thm}
Let $p, q, r>1,\ k \in [0,1)$ and $x \in [0,K_{p,q,r}(k)]$. 
Then, the following claims
$(\mathrm{i})$ and $(\mathrm{ii})$ hold true\,$:$
\begin{enumerate}
\item If $p \geq 2$,  then
\begin{equation*}
(\sn_{p,q,r}^{q/p}{(x,k)}+\cn_{p,q,r}{(x,k)})^{p^*}
\leq 1+2^{2/p}\sn_{p,q,r}^{q/p}{(x,k)}\cn_{p,q,r}{(x,k)}.
\end{equation*}
This inequality is reversed if $p \in (1,2]$.
\item If $q \geq 2$,  then
\begin{equation*}
(\sn_{p,q,r}{(x,k)}+\cn_{p,q,r}^{p/q}{(x,k)})^{q^*}
\leq 1+2^{2/q}\sn_{p,q,r}{(x,k)}\cn_{p,q,r}^{p/q}{(x,k)}.
\end{equation*}
This inequality is reversed if $q \in (1,2]$.
\end{enumerate}
\end{thm}

\begin{proof}
(i) Suppose that $p \geq 2$. Letting
\[\alpha=\frac{\sn_{p,q,r}^{q/p}{(x,k)}}{\sn_{p,q,r}^{q/p}{(x,k)}+\cn_{p,q,r}{(x,k)}}, \quad x \in [0,K_{p,q,r}(k)],\]
in \eqref{eq:liebineq} gives the assertion.

(ii) Suppose that $q \geq 2$. Letting
\[\alpha=\frac{\cn_{p,q,r}^{p/q}{(x,k)}}{\sn_{p,q,r}{(x,k)}+\cn_{p,q,r}^{p/q}{(x,k)}}, \quad x \in [0,K_{p,q,r}(k)],\]
in \eqref{eq:liebineq} with $p$ replaced by $q$ gives the assertion.
\end{proof}

In particular, letting $k=0$, we obtain the following inequalities
for GTFs with two parameters.
Moreover, letting $p=q$ implies Theorem \ref{thm:elineq}.

\begin{cor}
\label{cor:s+c}
Let $p, q>1$ and $x \in [0,\pi_{p,q}/2]$. 
Then, the following claims
$(\mathrm{i})$ and $(\mathrm{ii})$ hold true\,$:$
\begin{enumerate}
\item If $p \geq 2$,  then
\begin{equation*}
(\sin_{p,q}^{q/p}{x}+\cos_{p,q}{x})^{p^*}
\leq 1+2^{2/p}\sin_{p,q}^{q/p}{x}\cos_{p,q}{x}.
\end{equation*}
This inequality is reversed if $p \in (1,2]$.
\item If $q \geq 2$,  then
\begin{equation*}
(\sin_{p,q}{x}+\cos_{p,q}^{p/q}{x})^{q^*}
\leq 1+2^{2/q}\sin_{p,q}{x}\cos_{p,q}^{p/q}{x}.
\end{equation*}
This inequality is reversed if $q \in (1,2]$.
\end{enumerate}
\end{cor}


\section*{Acknowledgments}
The authors would like to thank Professor David Edmunds for the information on Theorem \ref{thm:elineq}.
The authors would also like to thank the anonymous referees for their help in improving the quality of the paper.



\end{document}